\documentclass[12pt]{amsart}
\usepackage{amssymb,latexsym,amsmath,amscd,amsthm,graphicx, color}
\usepackage[all]{xy}
\usepackage{pgf,tikz}
\usepackage{mathrsfs}
\usetikzlibrary{arrows}
\definecolor{uuuuuu}{rgb}{0.26666666666666666,0.26666666666666666,0.26666666666666666}
\definecolor{xdxdff}{rgb}{0.49019607843137253,0.49019607843137253,1.}
\definecolor{ffqqqq}{rgb}{1.,0.,0.}

\definecolor{uuuuuu}{rgb}{0.26666666666666666,0.26666666666666666,0.26666666666666666}
\definecolor{qqwuqq}{rgb}{0.,0.39215686274509803,0.}
\definecolor{zzttqq}{rgb}{0.6,0.2,0.}
\definecolor{xdxdff}{rgb}{0.49019607843137253,0.49019607843137253,1.}
\definecolor{qqqqff}{rgb}{0.,0.,1.}
\definecolor{cqcqcq}{rgb}{0.7529411764705882,0.7529411764705882,0.7529411764705882}
\definecolor{sqsqsq}{rgb}{0.12549019607843137,0.12549019607843137,0.12549019607843137}

\theoremstyle{plain}

\newtheorem{theorem}[subsection]{Theorem}

\newtheorem{lemma}[subsection]{Lemma}

\newtheorem{prop}[subsection]{Proposition}

\theoremstyle{definition}

\newtheorem{cor}[subsection]{Corollary}

\newtheorem{remark}[subsection]{Remark}

\begin{document}
To appear, Journal of Mathematical Analysis and Applications 
\title{Geometric and Measure-Theoretic Shrinking Targets in Dynamical Systems}

\author{Joseph Rosenblatt}
\address{Department of Mathematics \\
University of Illinois\\
1409 W. Green Street \\
Urbana, Illinois 61801-2907, USA}
\email{rosnbltt@illinois.edu}

 \author{Mrinal Kanti Roychowdhury}
\address{School of Mathematical and Statistical Sciences\\
University of Texas Rio Grande Valley\\
1201 West University Drive\\
Edinburg, TX 78539-2999, USA}
\email{mrinal.roychowdhury@utrgv.edu}

\date{September, 2019}
\thanks{}

\begin{abstract} We consider both geometric and measure-theoretic shrinking targets for ergodic maps, investigating when they are visible or invisible.  Some Baire category theorems are proved, and particular constructions are given when the underlying map is fixed.  Open questions about shrinking targets are also described.
\end{abstract}
\maketitle

\pagestyle{myheadings}\markboth{Joseph Rosenblatt and Mrinal Kanti Roychowdhury}{Geometric and Measure-Theoretic Shrinking Targets in Dynamical Systems}

\section{Introduction}\label{intro}

We consider {\em shrinking targets}, both ones that are geometric and ones that are measure-theoretic.  At a high level what this means is that we have sets $C_n\subset X$ that are {\em shrinking} in diameter in the geometric case, or shrinking in measure in the measure-theoretic case.  The sequence $(C_n)$ of sets is our {\em target}.  In addition, we have an invertible map $\tau$ of $X$.  We are interested in knowing the size of the set of $x$ such that $\tau^n(x) \in C_n$ infinitely often, even though the target $(C_n)$ is shrinking.  This clearly requires an interplay of the nature of the map and how fast the targets are shrinking.

To be more specific, first consider a metric space $(X,d_X)$ which is also a probability space $(X,\beta,m)$.  We consider an invertible measure map $\tau$ on $X$.  We also consider sequences of powers $(m_n)$  and decreasing values $(\epsilon_n)$, usually with $\lim\limits_{n\to \infty} \epsilon_n = 0$.  A basic question is: what properties will allow for $d_X(\tau^{m_n}(x),y) \le \epsilon_n$ infinitely often, for a.e. $x \in X$, and a.e. (or every) $y\in X$? That is, for the closed metric balls $B_{\epsilon_n}(y)$, the ones centered at $y$ of radius $\epsilon_n$, we want to have a.e. $x\in \tau^{-m_n}(B_{\epsilon_n}(y))$ infinitely often.   This is what we call a {\em geometric shrinking target property}.

But now if we choose to ignore $y$ and replace the metric balls by a decreasing sequence $(B_n)$ of measurable sets, usually with their measures converging to zero, this becomes what we call a {\em measure-theoretic shrinking target property}.  In this case we are asking when would we have a.e. $x \in \tau^{-m_n}(B_n)$ for infinitely many $n$?

As for the underlying space $X$, we always assume that  $(X,\beta,m)$ is a non-atomic, separable probability space.  Hence, in particular the measure-theoretic shrinking target property can be viewed as a generalization of classical properties of an ergodic map.   Indeed, take $\tau \in \mathcal M$, the measure-preserving invertible maps of $X$, integers $(m_n:n\ge 1)$ and shrinking measurable sets $(B_n:n \ge 1)$.  But suppose that actually all $B_n$ are a fixed set $B$ with $m(B) > 0$.  Then a.e. visibility is just Poincar\`e recurrence: if $\tau$ is ergodic and
$B \in \beta$ with $m(B) > 0$, then for a.e. $x \in X$, one has $\tau^n(x) \in B$ for infinitely many $n\ge 1$.    That is, for all $N\ge 1$, $X = \bigcup\limits_{n=N}^\infty \tau^{-n}(B)$ (up to a null set).   This recurrence happens very easily because instead of a sequence $(B_n)$ of different sets with a null intersection, we have only one set of positive measure as our target.  On the other hand, if we did have a sequence of sets $(B_n)$, and  $\sum\limits_{n=1}^\infty m(B_n) < \infty$, then certainly $m(B_n)$ goes to zero.  But also for a.e. $x$, $\tau^{m_n}(x) \in B_n$ only finitely many times because $\int \sum\limits_{n=1}^\infty 1_{\tau^{-m_n}(B_n)}\, dm = \sum\limits_{n=1}^\infty m(B_n) < \infty$.   We want to understand in the measure-theoretic context what happens between these two extremes.

If in addition we assume that the space is a metric space, say concretely the interval $[0,1]$ with addition modulo one and the usual Euclidean distance, then as noted above we want to consider the geometric shrinking target property that for a.e. (or all) $y\in [0,1]$,   $|y - \tau^n(x)|\le \epsilon_n$ for infinitely many $n$, for a.e. $x\in [0,1]$.  The most famous general result of this type is by Boshernitzan~\cite{B}.  First take the target point $y$ to be self-referentially $x$ itself, and assume that $\tau \in \mathcal M$.  Boshernitzan showed that for a.e. $x$, $\liminf\limits_{n\to \infty} n|x - \tau^n(x)| \le 1$.

This geometric shrinking target property has been studied extensively for rotations of the circle originally, and more recently other classes of maps like interval exchange transformations (IETs).  What has come out in this work for rotations of the circle is that,  depending on the properties of the terms in the standard continued fraction expansion $\theta\in [0,1]$, there will be an optimal rate $\epsilon_n$ decreasing to $0$ such that for all $y\in [0,1]$,  $|y- \{n\theta + x\}|\le \epsilon_n$ for infinitely many $n$, for a.e. $x\in [0,1]$.  For example, see Simmons~\cite{Simmons} and Tseng~\cite{T}.

The type of dynamical system phenomena sketched above, both for measure-theoretic and geometric shrinking targets, are the focus of this article.  We are the most interested in properties of the map $\tau$ and the shrinking target $(B_n)$, whether geometric or measure-theoretic, that make it {\em a.e. visible} with respect to the map.  As the extreme failure of this, we also want to know in both cases when the shrinking target is {\em a.e. invisible}.  See Section~\ref{adjectives} for terminology, in particular definitions of visible and invisible.

For example, in Section~\ref{Quant}, we show how quantization using orbits of the map can explicitly give geometric shrinking targets for the map.  Then in Section~\ref{generic}, we prove some Baire category results that show what can be said at least generically about visibility of shrinking targets given rates that the diameters or measures of the targets decrease to zero.  Lastly, in Section~\ref{particular}, we consider what happens the map is fixed.  For example, we show that there is an a.e. visible measure-theoretic shrinking target with any predetermined rate $(\epsilon_n)$ decreasing to zero, with of course the necessary constraint that $\sum\limits_{n=1}^\infty \epsilon_n = \infty$.  But also we show using a covering rate result, that no matter how slowly $\epsilon_n > 0$ tends to zero, there is an a.e. invisible measure-theoretic shrinking target $(B_n)$ for $\tau$ with $m(B_n) \ge \epsilon_n$ for all $n$.

In one form or another, the topic of shrinking targets has deservedly received a lot of attention from a number of authors.  Some of this work is geometric.  Some of it is measure-theoretic, for example particularly applying measure theory to the behavior of maps on $[0,1]$ that arise in Diophantine approximation.  It would be difficult to give credit to everyone for ideas that suggested the language and results in this article.  But we should certainly cite some of the articles of a number of authors on gap theorems and shrinking target properties.  These include the paper by Kurzweil~\cite{K} and the work by  Philipp in ~\cite{P1,P2,P3}.  Also, there is a very useful survey on the topic of shrinking targets by Athreya~\cite{A}.  In addition, see the article by Chernov and Kleinbock~\cite{CK}.  More recently, on the topic of shrinking targets in geometrical settings, there is the article by Kleinbock and Zhao~\cite{KZ}, which has been improved on in the article by Kelmer~\cite{K}, and then by Kelmer and Yu~\cite{KY}.

\section{Notation and Language}\label{adjectives}

Here is some descriptive language we would like to propose.  We think it will help clarify the nature of the results that follow.

Any sequence $(B_n)$ of measurable sets is called a {\em target}, albeit in some sense the sets themselves are really the targets.  Any decreasing sequence $(B_n)$, with $\lim\limits_{n\to \infty} m(B_n) = 0$, is called a {\em measure-theoretic shrinking target}.   If $(B_n)$ is a measure-theoretic shrinking target and we have $\tau^n(x) \in B_n$ infinitely often for all $x\in C$ with $m(C) > 0$, then we say $(B_n)$ is a {\em visible measure-theoretic shrinking target}.  One might say here that the visibility is with respect to $C$ since one might not be able to assert this for larger sets.   However, if indeed this happens with $m(C) = 1$, then we say that we have an {\em a.e. {\bf visible} measure-theoretic shrinking target}.  At the other extreme, if for a.e. $x$, we have $\tau^n(x) \in B_n$ for {\bf only} finitely many $n$, then we say that $(B_n)$ is an {\em a.e. {\bf invisible} measure-theoretic shrinking target}.

We are also interested in these same properties where we restrict the set of powers to some increasing sequence of whole numbers $(m_n: n\ge 1)$. When this is the context, we will refer to the target property relative to $(m_n)$.  As observed above, for any ergodic $\tau$ and any target $(B_n)$ with all $B_n$ being a fixed set $B$ of positive measure, we have an a.e. visible measure-theoretic target (although not a shrinking one) relative to the whole numbers $\mathbb N$.  On the other hand, when  $\sum\limits_{n=1}^\infty m(B_n) < \infty$, then $(B_n)$ is an a.e. invisible measure-theoretic target relative to $(m_n)$ for any map $\tau$ and any sequence of powers $(m_n)$.  The shrinking condition is not needed here.

 If in addition we specifically are considering sets $B_n$ that are closed balls $B_{\epsilon_n}(y)$ with respect to an underlying metric, we will add the adjective ``geometric'' to our various cases, but drop the adjective ``measure-theoretic''.     Note: we consider non-atomic measures, so geometric shrinking targets will also be measure-theoretic shrinking targets.  As an example of this terminology, suppose we have an ergodic map $\tau$ of $[0,1]$ and can prove that
for a.e. $x\in [0,1]$ and every $y\in [0,1]$, $\liminf\limits_{n\to \infty}n|y - \tau^n(x)|\le 1$.  Then we would have the intervals $([-1/n +y,1/n+y]:n\ge 1)$ being an a.e. visible geometric shrinking target for all $y$.  That is, we have a family of a.e. visible geometric shrinking targets indexed by $y\in [0,1]$.

All of these properties depend explicitly on the map $\tau$, so technically
we should include this in the language.  We will do this when the emphasis seems important, but we will typically leave this as simply understood from the context.
\section{Connections to Quantization}\label{Quant}

The dynamical approach to quantization taken in Rosenblatt and Roychowdhury~\cite{RR} is to generate quantizers using ergodic maps of $[0,1]$ with the usual Lebesgue measure $m$.  These quantizers will intersect shrinking targets consisting of intervals, if the diameter of the sets in the target do not shrink too quickly.  In this way, a dynamical approach to quantization connects with questions and results about shrinking targets in dynamical systems, a topic that as we have noted above has been studied fairly extensively for Diophantine, dynamical, and geometric models.

Here is how the quantization results in Rosenblatt and Roychowdhury~\cite{RR} give very specific results on how quantizers can play a role in providing facts about shrinking interval targets.  First, besides other quantization distortion error rates, in ~\cite{RR} we consider the {\em geometric distortion error} $r_n^\tau (x)$.  The value $r_n^\tau(x)$ is the minimum radius $r$ such that the finite number of closed intervals $[\tau^k(x) - r, \tau^k(x) +r], k=1,\dots,n$, cover $[0,1]$ (modulo one).  When $\tau$ is ergodic, $r_n^\tau (x) \to 0$ as  $n\to \infty$ for a.e. $x$.  Hence, one can show there is $\rho_n$ decreasing to $0$ such that for a.e. $x$, $r_n^\tau(x) \le \rho_n$ for large enough $n$, depending on $x$.  We know what this rate is for some specific types of mappings, but we do not know what it is in general.  Also, studying this same covering rate in more general geometric settings is very worthwhile.  On the interval, it is clear though that $r_n^\tau(x) \ge 1/2n$ in any case, and so there is a lower bound of this size eventually.

Now let $B_g(y)$ be the interval of length $2g$ centered at $y$.   What we have by the definition of the geometric distortion error is that for a.e. $x$, if $n$ is large enough, then $[0,1] \subset  \bigcup\limits_{k=1}^n B_{\rho_n}(\tau^k(x))$.   Hence, splitting up the cover, we have
\[[0,1]\backslash \bigcup\limits_{k=1}^{K-1} B_{\rho_{K+N}}(\tau^k(x)) \subset \bigcup\limits_{k=K}^{N+K} B_{\rho_{K+N}}(\tau^k(x)).\]
We can then adjust the radii on the right, and let $N$ go to infinity.  With an application of Boshernitzan's Theorem ~\cite{B}, we can show the following.

\begin{prop}  For every ergodic mapping $\tau$, there is a sequence $(\rho_n)$ tending to zero so that for all $y\in [0,1]$, the intervals $B_{\rho_n}(y)$ are an a.e. visible geometric shrinking target with respect to $\tau$.
\end{prop}
\begin{proof}  Again, we begin with a decreasing sequence $\rho_n$ with $\lim\limits_{n\to \infty} \rho_n = 0$.  We assume that we have for a.e. $x$, $r_n^\tau(x) \le \rho_n$ eventually.  This means at least $\rho_n \ge 1/2n$ eventually.  But we adjust this by scaling this by $6$, so that we actually have $\rho_n \ge 3/n$ eventually.

Now fix $x$ as above.  Then for any $K$, if $N$ is large enough, we have
\begin{align*} [0,1]\backslash \bigcup\limits_{k=1}^{K-1} B_{\rho_{K+N}}(\tau^k(x)) &\subset \bigcup\limits_{k=K}^{N+K} B_{\rho_{K+N}}(\tau^k(x)) \subset \bigcup\limits_{k=K}^{N+K}  B_{\rho_k}(\tau^k(x))
\subset \bigcup\limits_{k=K}^\infty B_{\rho_k}(\tau^k(x)).
\end{align*}
However, as $N \to \infty$, $\bigcup\limits_{k=1}^{K-1}\left (B_{\rho_{K+N}}(\tau^k(x))\right )$ decreases to $\{\tau(x),\dots,\tau^{K-1}(x)\}$.  Hence, we have any $y$ not in the orbit $\{\tau^k(x):k\ge 1\}$ is in every tail set $\bigcup\limits_{n=K}^\infty B_{\rho_n}(\tau^n(x))$.  That is, for any $y$ not in the orbit $\{\tau^k(x): k \ge 1\}$, we have $|y - \tau^n(x)| \le \rho_n$ infinitely often.

In addition, Boshernitzan~\cite{B} shows that for a.e. $x$, and any $l\ge 1$, we have $|\tau^{k+l}(x) - \tau^l(x)| \le 2/k$ for infinitely many $k$.  Hence, for a.e. $x$ and all $l\ge 1$, $|\tau^{k+l}(x) - \tau^l(x)|\le 3/(k+l) \le \rho_{k+l}$ for infinitely many $k$.  Combining the covering result above with this, we see that for a.e. $x$, we have for any $y$, the intervals $[y-\rho_n,y+ \rho_n]$ contain $\tau^n(x)$ for infinitely many $n$.  This gives the a.e. visible geometric shrinking target property we wanted.
\end{proof}

\begin{remark}\label{limitations} Unfortunately, the actual best values of $r_n^\tau$ are not known in general.  Indeed, in \cite{RR}, sometimes just to get an estimate, we have to use the (generally) much larger discrepancy of the sequence $(\tau^k(x): k\ge 1)$ as a proxy for the geometric distortion error.  On the other hand, there are important cases where the best possible results are known.  Take as an example $\tau_\theta$ to be rotation in the circle by an angle $\theta$ with terms in its continued fraction decomposition bounded by $C_o$.  Then $\rho_n = 1/n$.   But this rate result does not hold in general, it needs to be slower for other rotations.  Now such a rate result gives the geometric shrinking target result of Kim~\cite{K1,K2}, but only for this particular type of rotation.  However, the same rate $1/n$ is shown by Kim to hold for ergodic rotations in general by other methods.  So it is clear that there is generally a loss of speed when using the geometric distortion error to derive a geometric shrinking target result.
\end{remark}

\begin{remark}\label{Constraints}
The basic questions that have been considered for shrinking targets are extensions of point or set recurrence of the dynamical system.  In Boshernitzan's Theorem 2.1~\cite{B}, for any measure-preserving map  $\liminf\limits_{n\to \infty} n|x - \tau^n(x)| \le 1$ for a.e. $x$.
One might hope that more generally, if $\tau$ is ergodic then for any $y$, we would have also for some bounded function $C(x)$,  $\liminf\limits_{n\to \infty} n|y - \tau^n(x)| \le C(x)$ for a.e. $x$.  This is proved to be the case for irrational rotations in Kim~\cite{K1,K2}; indeed, it is shown in that case that $C = 0$.  But we have to anticipate that this type of very explicit result perhaps does not hold more generally.  For example, consider the proposed Baire category result Theorem B in Junqueira~\cite{J}; we will come back to this issue later in this article.
\end{remark}

\begin{remark}\label{Shepp} If we want a.e. visible shrinking targets for a random sequence of intervals $(I_n)$ in $[0,1]$, Levy's theorem shows that lengths $|I_n| = \log n/n$ work.  But actually all one needs is $\sum\limits_{n=1}^\infty |I_n| = \infty$ if the goal is to cover the interval infinitely often up to a null set.  However, Shepp's Theorem~\cite{S} gives a best possible result if we want the stronger condition that the random intervals cover $[0,1]$ completely.  It says that a uniformly chosen random sequence of intervals $(I_n)$ a.s. covers the circle completely if and only if $\sum\limits_{n=1}^\infty \frac {\exp (|I_1|+\dots+|I_n|)}{n^2} = \infty$.  So lengths $|I_n| = \log n/n$ from Levy's theorem work but are too large since also the lengths $|I_n| = 1/n$ suffice.  However, it is interesting that $|I_n| = \delta/n,\delta < 1$ does not suffice for this even though still $\sum\limits_{n=1}^\infty |I_n| = \infty$.
\end{remark}

\section{Generic Results}\label{generic}

An important point that needs to be made is that  the sets of maps giving quantitative descriptions of shrinking targets and the sets of maps giving geometric distortion error rates have a very different character in terms of descriptive set theory.  Simply put, just because a set $S$ is dense in some topology, we cannot necessarily extend that to a Baire category result unless we can give the appropriate set theoretic description of $S$.  This is technically what leads to different types of generic results for geometric distortion error rates and for shrinking target theorems.

To see further what the issue is, note that in \cite{RR} it is shown that any quantization error rate fixed at the outset will be violated by the generic ergodic map.  However, we believe that the following Theorem~\ref{TypShrink} holds for more general maps than just the measure-theoretic ones.  Although what is proved is for measure-preserving maps only, this result is still not the same as the one proposed in Theorem B in Junqueira~\cite{J}.  Some clarification of this contrast in results (one might say conflict) needs to be given.

Now, in the following Theorem~\ref{TypShrink}, we could use a rate larger than all of the geometric distortion errors given by some dense sequence of maps to get a shrinking rate that applies to a dense set of maps.  However, there actually are better results available.  We use instead the result in Chaika~\cite{C} which specifically is about shrinking target theorems.  This gives the result below that the generic map has an a.e. visible geometric shrinking targets with respect to a fixed sequence of radii $(\epsilon_n)$.

\begin{theorem}\label{TypShrink}  Take any decreasing sequence $(\epsilon_n)$ with $\epsilon_n > 0$ for all $n\ge 1$, and such that $\sum\limits_{n=1}^\infty \epsilon_n = \infty$.  Consider the set $\mathcal R$ of ergodic mappings $\tau$, such that we have the geometric shrinking target property that for a.e. $y$ and for all $N\ge 1$, up to a null set
\[[0,1] = \bigcup\limits_{n=N}^\infty \tau^{-n}(B_{\epsilon_n}(y)).\]
The set $\mathcal R$ is a dense $G_\delta$ set in $\mathcal M$.
\end{theorem}
\begin{proof}
Corollary 1 in Chaika~\cite{C} tells us that there is a  dense set $\mathcal D$ of ergodic maps $\tau$ (actually IETs) such that the geometric shrinking target property above holds.

So, now consider the set $\mathcal G$ of mappings $\tau$ such that for a.e.  $x$ and for all $N$,
\[m\left ( \bigcup\limits_{n=N}^\infty B_{\epsilon_n}(\tau^n(x))\right ) = 1.\]
This is exactly the set of mappings $\mathcal R$.  Indeed, $\tau \in \mathcal G$ if and only if for a.e. $x$, a.e. $y$ is in this tail set for all $N$.  That is, for a.e. $x$, we have for a.e. $y$,  $|y - \tau^n(x)|\le \epsilon_n$ infinitely often.  Equivalently, for a.e. $y$, up to a null set,
\[[0,1] = \bigcup\limits_{n=N}^\infty \tau^{-n}(B_{\epsilon_n}(y)).\]
Hence, $\mathcal R = \mathcal G$.

We can express $\mathcal G =\bigcap\limits_{s=2}^\infty\bigcap\limits_{r=2}^\infty \bigcap\limits_{N=1}^\infty \bigcup\limits_{M\ge N} \mathcal G(r,s,N,M)$ where $\mathcal G(r,s,N,M)$ consists of all $\tau$ such that
$m\left (\bigcup\limits_{n=N}^M B_{\epsilon_n}(\tau^n(x))\right ) > 1 - \frac 1s $ on a set of $x$ of measure greater than $1 - \frac 1r$.  We see that
$\mathcal D \subset \bigcup\limits_{M\ge N}
\mathcal G(r,s,N,M)$.  But also, we claim that each $\mathcal G(r,s,N,M)$ is open in the metric topology on $\mathcal M$.  Hence,  $\bigcup\limits_{M\ge N} \mathcal G(r,s,N,M)$ is an open dense set, and so $\mathcal G$ is a dense $G_\delta$ set.

To show that $W = \mathcal G(r,s,N,M)$ is open, it suffices to show that if $\tau \in W$ and $\tau_j \to \tau$ in the weak topology, as $j \to \infty$, then there is some term $\tau_j \in W$ also.  This characterization of being an open set follows immediately from using the usual metric for the weak topology.

So take $\tau \in W$ and a sequence $(\tau_j:j\ge 1)$ converging to $\tau$ in the weak topology.  Then there is a subsequence, $(\tau_{j_i}:i\ge 1)$ such that $\tau_{j_i} \to\tau$ a.e. as $i\to \infty$.  Indeed, the subsequence can be chosen so that
$\tau^n_{j_i} \to\tau^n$ a.e. as $i \to \infty$, for all $n, N\le n \le M$.
But then it clearly follows that as $i\to \infty$, for a.e. $x$,

\[m\left (\bigcup\limits_{n=N}^M B_{\epsilon_n}(\tau_{j_i}^n(x))
\right )
\to m\left (\bigcup\limits_{n=N}^MB_{\epsilon_n}(\tau^n(x))\right ).\]
But convergence a.e. implies convergence in measure.  Hence, for some large enough $i$,
\[m\left (\bigcup\limits_{n=N}^MB_{\epsilon_n}(\tau_{j_i}^n(x))\right ) > 1 - \frac 1s\]
on a set $E$ with $m(E) > 1 - \frac 1r$, just as was the case for $\tau$ because it is in $W$.  Hence, for such $i$, $\tau_{j_i} \in W$ too.
\end{proof}

It would be better to prove this result for all $y$.  But it is not clear if this is true.  However, if we modify the expectation and just show that this set of maps is residual i.e. its complement is meager, then this is true.  Taking this approach leads to some tricky issues with the order of the quantifiers in the results that follow.  This may be just a consequence of the methods that we use and can be eliminated with better arguments.  First, we have the following.

\begin{theorem}\label{TypShrinkAll}  Take any decreasing sequence $(\epsilon_n)$ with $\epsilon_n > 0$ for all $n\ge 1$, and such that $\sum\limits_{n=1}^\infty \epsilon_n = \infty$.  Consider the set of ergodic mappings $\tau$, such that we have the geometric shrinking target property: there is a set of full measure $G$ such that for all $y$ and for all $N\ge 1$,
\[G \subset \bigcup\limits_{n=N}^\infty \tau^{-n}(B_{\epsilon_n}(y)).\]
Then this set contains a dense $G_\delta$ set in $\mathcal M$.
\end{theorem}
\begin{proof}  Corollary 1 in Chaika~\cite{C} actually tells us that there is a  dense set $\mathcal D$ of ergodic maps $\tau$ (actually IETs) such that the geometric shrinking target property above holds taking $\epsilon_n/2$ in place of $\epsilon$.

We make a technical modification in the proof of Theorem~\ref{TypShrink} to prove this result.  Let $B_\epsilon^o(z)$ be the interior of $B_\epsilon(z)$, i.e. in this case the open interval instead of the closed interval. Now consider the set $\mathcal G$ of mappings $\tau$ such that for a.e.  $x$ and for all $N$,
\[[0,1] =
\bigcup\limits_{n=N}^\infty B_{\epsilon_n}^o(\tau^n(x)).\]
\noindent That is, consider all $\tau$ such that for any whole number $r \ge 1$, we have for a set of $x$ of measure greater than $1 - 1/r$, for every $N\ge 1$,
$[0,1] = \bigcup\limits_{n=N}^\infty B_{\epsilon_n}^o(\tau^n(x)).$

Now, take such a map $\tau$.
For a fixed $x$, and $M \ge N$, both fixed,
the set of $y$ that are in
$\bigcup\limits_{n=N}^M B_{\epsilon_n}^o(\tau^n(x))$
is an open set.  Hence, for a.e. $x$, by compactness, if $N$ is fixed and $M$ varies, then there is some $M\ge N$ such that actually
$[0,1] = \bigcup\limits_{n=N}^M B_{\epsilon_n}^o(\tau^n(x)).$
But then this shows that we are actually considering the maps
 $\tau$ such that for any whole number $r \ge 1$, we have for a set of $x$ of measure greater than $1 - 1/r$, for every $N\ge 1$, there is $M \ge N$ such that
$[0,1] =\bigcup\limits_{n=N}^M B_{\epsilon_n}^o(\tau^n(x)).$
That is, $\mathcal G =\bigcap\limits_{r=2}^\infty \bigcap\limits_{N=1}^\infty \bigcup\limits_{M\ge N} \mathcal G(r,N,M)$ where $\mathcal G(r,N,M)$ consists of all $\tau$ such that
on a set of $x$ of measure greater than $1 - \frac 1r$, we have
$[0,1] = \bigcup\limits_{n=N}^M B_{\epsilon_n}^o(\tau^n(x)).$
But also, we claim that each $\mathcal G(r,N,M)$ is open in the metric topology on $\mathcal M$.  Hence,  $\bigcup\limits_{M\ge N} \mathcal G(r,N,M)$ is an open set, which contains $\mathcal D$, and so $\mathcal G$ is a dense $G_\delta$ set.

To show that $W = \mathcal G(r,N,M)$ is open, it suffices to show that if $\tau \in W$ and $\tau_j \to \tau$ in the weak topology, as $j \to \infty$, then there is some term $\tau_j \in W$ also.
By passing to a subsequence, we may assume without loss of generality $(\tau_j^n:j\ge 1)$ converges a.e. to $\tau^n$ for all $n, N\le n \le M$.  Say this a.e. behavior occurs on a set of full measure $G$.  But then consider a fixed $y
\in  \bigcup\limits_{n=N}^M B_{\epsilon_n}^o(\tau^n(x))$.  That is,
$|\tau^n(x) - y |  <\epsilon_n$ for some $n, N\le n\le M$.  Hence, for $x\in G$, and
for large enough $j$, $|\tau_j^n(x) - y |  <\epsilon_n$ for some $n, N\le n\le M$.   So for $x\in G$, $y \in \bigcup\limits_{n=N}^M B_{\epsilon_n}^o(\tau_j^n(x))$ for some large enough $j$.  But the sets $\bigcup\limits_{n=N}^M B_{\epsilon_n}^o(\tau_j^n(x))$ are open sets, so again by compactness of $[0,1]$, for $x\in G$, there is some large enough $j$ such that
$[0,1] = \bigcup\limits_{n=N}^M B_{\epsilon_n}^o(\tau_j^n(x))$.
But then if we know that for a set of $x$ of measure greater than $1 - 1/r$, $[0,1] =\bigcup\limits_{n=N}^M B_{\epsilon_n}^o(\tau^n(x))$, we would actually have for a large enough $j$, on a set of $x$ of measure greater than $1 - 1/r$,
$[0,1] =\bigcup\limits_{n=N}^M B_{\epsilon_n}^o(\tau_j^n(x))$.
Thus, $\tau_j \in W$ for some $j$.  Hence, $W$ is open and the proof is complete.
\end{proof}

The geometric shrinking target phenomenon above is based on finding maps and typical orbits that enter all, or almost all, shrinking targets that are intervals, and not general measurable sets. One might be able to obtain a similar result for a given map, or a class of maps (e.g. a generic set in $\mathcal M$), by abstracting the process.  That is,  instead of focusing on geometric shrinking targets, consider shrinking the targets that are just measurable sets.  Using the language we introduced above, this change is a switch from the geometric shrinking target property to the measure-theoretic shrinking target property.  So take a decreasing sequence of measurable sets $(B_n)$ instead of intervals.     Then for which maps, if any, is it the case that $\bigcup\limits_{n = N}^\infty \tau^{-n}(B_n) = [0,1]$ a.e. for all $N \ge 1$?   This property requires  $\sum\limits_{n=1}^\infty m(B_n) = \infty$ by the trivial side of the Borel-Cantelli Lemma.

We will see in Theorem~\ref{MeasTypShrink} that the measure-theoretic shrinking target property above can be written as a $G_\delta$ set in $\mathcal M$ with the usual weak topology.  So the phenomenon would be generic if there is a dense class of maps with this property.  But actually Corollary 1 in Chaika~\cite{C} also gives this too.  See also the fundamental paper by Kurzweil~\cite{K}.  Fix a decreasing sequence of measurable sets $(B_n)$ with $m(B_n) > 0$ for all $n\ge 1$, and $\sum\limits_{n=1}^\infty m(B_n) = \infty$.  We can take a measure preserving map $\sigma$ such that $\sigma (B_n) = [0,m(B_n)]$ for all $n\ge 1$.  Then in Corollary 1 take the case that $y = 0$ from the definition of the strong Kurzweil property.  We get a dense set of maps (particular IETs again) such that $\bigcup\limits_{n=N}^\infty \tau^{-n}\sigma (B_n) = [0,1]$ up to a null set, for any $N \ge 1$.  Hence, the maps $\sigma^{-1}\circ \tau\circ \sigma$, with $\tau$ being a.e. IET as in Corollary 1 in \cite{C}, give the dense set $\mathcal D$ we need to get a category result.  This gives a very general result that the generic map has $(B_n)$ as an a.e. visible measure-theoretic shrinking target.

\begin{theorem}\label{MeasTypShrink}  For any decreasing sequence of measurable sets $(B_n)$ with $m(B_n) > 0$ for all $n\ge 1$, and $\sum\limits_{n=1}^\infty m(B_n) = \infty$,  there is a dense $G_\delta$ set of ergodic maps $\tau$ such that up to a null set,
\[\bigcap\limits_{N =1}^\infty \bigcup\limits_{n=N}^\infty \tau^{-n}(B_n) = [0,1].\]
\end{theorem}
\begin{proof}  Consider the set $\mathcal G$ of mappings $\tau$ such that for all $N \ge 1$,
\[m\left ( \bigcup\limits_{n=N}^\infty \tau^{-n}(B_n)\right ) = 1.\]
We can express $\mathcal G =\bigcap\limits_{s=2}^\infty \bigcap\limits_{N=1}^\infty \bigcup\limits_{M\ge N} \mathcal G(s,N,M)$ where $\mathcal G(s,N,M)$ consists of all $\tau$ such that
$m\left (\bigcup\limits_{n=N}^M \tau^{-n}(B_n)\right ) > 1 - \frac 1s $.
But in the weak topology on $\mathcal M$, as $\tau_j\to \tau$, we would have
\[m\left (\bigcup\limits_{n=N}^M \tau_j^{-n}(B_n)\right ) \to
m\left (\bigcup\limits_{n=N}^M \tau^{-n}(B_n)\right ) .\]
So $\mathcal G(s,N,M)$ is open in the weak topology, and therefore so is
$\bigcup\limits_{M\ge N} \mathcal G(s,N,M)$.  Hence, $\mathcal G$ is a $G_\delta$ set.  The comments above show how the results in Chaika~\cite{C} prove that it is also dense.
\end{proof}

\begin{remark} a) It is worth observing that the choice of $(B_n)$ and $\tau$ here have certain inherent mutability.  Indeed, suppose $(B_n)$ and $(C_n)$ are decreasing sequences of measurable sets with $m(B_n) = m(C_n)$ for all $n\ge 1$.  Then there is an element $\sigma \in \mathcal M$ such that $\sigma (B_n) = C_n$ for $n$.  To do this, let $B_o = C_o = X$, and take invertible, measure-preserving maps $\sigma_n$ such that $\sigma_n(B_n\backslash B_{n+1}) = C_n\backslash C_{n+1}$ for all $n\ge 0$.  Then let $\sigma$ be $\sigma_n$ on $B_n\backslash B_{n+1}$.  Also, take $\sigma_\infty$ to be any invertible, measure-preserving map from $\bigcap\limits_{n=1}^\infty B_n$ onto $\bigcap\limits_{n=1}^\infty C_n$.  These choices define an invertible, measure-preserving map $\sigma$ on $X$ such $\sigma(B_n) = C_n$ for all $n,0\le n \le \infty$.  Now, if $(B_n)$ is a.e. visible with respect to $\tau$, then $(C_n)$ is a.e. visible with respect to $\omega = \sigma\circ\tau\circ\sigma^{-1}$.  But $\tau$ and $\omega$ are isomorphic.  So we can exchange an a.e. visible measure-theoretic shrinking target, with respect to $\tau$, for any other measure-theoretic shrinking target with the same measure-theoretic footprint and preserve a.e. visibility if we switch to another map, actually to an isomorphic copy of $\tau$.
\medskip

\noindent b) Take a sequence of measurable sets $(B_n)$ with $\sum\limits_{n=1}^\infty m(B_n) = \infty$.  They could be assumed to be nested, i.e. decreasing, or not.  We wonder if the shrinking target property is still generic if one specifies that the powers come from a particular subsequence of $\mathbb N$.  That is, we fix $(m_n)$ increasing, and then seek maps $\tau$ such that up to a null set
 \[\bigcup\limits_{n = N}^\infty \tau^{-m_n}(B_n) = [0,1]\]
 for all $N$.
Is this typical or can it be first category with a suitable fixed choice of $(m_n)$?  The same issue comes up with the geometric shrinking target property.  Is it typical that a map $\tau$ has the property that for a.e. $x$ and all $y$ (or say just a.e. $y$), we have  $\tau^{m_n}(x) - y \in B_n \mod 1$ infinitely often.  We would like this to happen generically in $\tau$.  The case $m_n = n$ is the original case.    Of course, we  need to assume here that $\sum\limits_{n=1}^\infty m(B_n) = \infty$.
But if one switches to a recurrence phenomenon instead of a density phenomenon perhaps this no longer matters.  Indeed, it would be very interesting if it is a generic phenomenon that for any  $\epsilon_n  > 0$,
$\liminf\limits_{n\to \infty}  \frac 1{\epsilon_n} |x - \tau^{m_n}(x)| \le 1$ for a.e. $x$.
\end{remark}

\begin{remark}  It is possible to use almost invariant sets to produce examples that restrict the category results above.  We will be more specific about this in the next section.  But here is roughly the idea.  One first fixes an ergodic map $\tau$.  Then choose $(\delta_n: n \ge 1)$ decreasing to zero.  One constructs explicitly a decreasing sequence of measurable sets $(B_n:n\ge 1)$ such that $m(B_n) \ge \delta_n$ for all $n\ge 1$, and $m\left (\bigcup\limits_{n=1}^\infty \tau^{-n}(B_n)\right ) < 1$.  So $C = X\backslash \bigcup\limits_{n=1}^\infty \tau^{-n}(B_n)$ has positive measure, and $\tau^n(x) \notin B_n$ for every $x\in C$.  Thus $(B_n)$ is a slowly shrinking target that is still not visible to points in $C$.  Hence, while the generic map $\tau$ has the a.e. visible measure-theoretic shrinking target property for $(B_n)$, not every map does.  In fact, no matter how slowly the measures $m(B_n)$ are made to decrease to zero, one can show that there can be a map $\tau$ that fails to have the a.e. visible measure-theoretic shrinking target property for the specific  $(B_n)$ that are constructed for $\tau$.  It might be possible to carry out this construction so that the set of maps are actually dense, instead of being just one map.
\end{remark}

\section{Shrinking Targets for Particular Maps}\label{particular}

In Adams and Rosenblatt~\cite{AR}, the goals were different than the ones in this article.  They were to consider various aspects of functions that are coboundaries with respect to a set of maps.  But in the process, there were constructions, particularly ones of sets that slowly fill out space, that we can use to get interesting information about measure-theoretic shrinking targets for a given fixed map.  We give first a brief explanation of some of the relevant results and methods in \cite{AR}, and then turn to applying them to give results about measure-theoretic shrinking targets.

A $\tau$-coboundary is a function $f\in L_r(X)$ such that there is some $F\in L_s(X)$ such that $f = F - F\circ \tau$.  The function $F$ is called the {\em transfer function}.  It is well-known that for $\tau$ ergodic, with $1\le r=s < \infty$, the set of $\tau$-coboundaries is first category.  The following is a less known fact about coboundaries from \cite{AR}.

\begin{prop}\label{mostfcnsnotcob} Assume $\tau$ is ergodic.
The generic function $f \in L_r(X), 1 \le r \le \infty$ is not a $\tau$-coboundary with
a measurable transfer function.
\end{prop}

The Baire category statement in Proposition~\ref{mostfcnsnotcob}, and other results in ~\cite{AR}, show only indirectly how to construct functions that are not coboundaries.  For this reason, one might want a better understanding of how to construct such functions directly.

Here is a simple, basic example.
Suppose we construct a set $E$ such that $m(\bigcup\limits_{k=1}^n \tau^{-k} (E)) < 1$ for all $n$.  Let $F = X\backslash E$.  Then consider $f = 1_F - m(F)$.  This $f$ is a bounded, mean-zero function.  Taking  $S_n^\tau f$, to be $\sum\limits_{k=1}^n f \circ \tau^k$ for any $n$, we have $S_n^\tau f = n - np(F)$ on $\bigcap\limits_{k=1}^n \tau^{-k} (F)$.  Since
$m(\bigcup_{k=1}^n \tau^{-k} (E)) < 1$, we have $m(\bigcap\limits_{k=1}^n \tau^{-k}(F)) > 0$ for all $n$. Hence,
$\|S_n^\tau f\|_\infty \ge np(E)$ and so, by a well-known principle discussed in ~\cite{AR}, $f$ is not a $\tau$-coboundary with transfer function in $L_\infty(X)$.

For the purposes of this article, there is another fact that such sets $E$ gives us.  Take $B_n = X\backslash\bigcup_{k=1}^n \tau^{-k} (E)$.  Then $(B_n)$ is a shrinking target and $\tau^n(B_n)$ is disjoint from $E$ for all $n$.  Hence, $(B_n)$ is not  a.e. visible with respect to $\tau^{-1}$.

Now there are many ways to construct such sets $E$.  For example, it is easy to show this result from \cite{AR}.

\begin{prop}\label{setgensmall}  The generic set $E$ has $m(\bigcup\limits_{k=1}^n \tau^{-k} (E)) < 1$ for all $n$.
\end{prop}

\begin{remark} The proof of Proposition~\ref{setgensmall} is easy and includes a  simple way to construct a particular set $E$ with the desired property.  Fix $\epsilon > 0$.  Take a sequence of sets $(E_n)$ with
$m(E_n) > 0$ for all $n$, and with $\sum\limits_{n=1}^\infty nm(E_n) \le
\epsilon$.  Then let $E = X\backslash\left (\bigcup\limits_{n=1}^\infty \bigcup\limits_{k=1}^n \tau^k(E_n)\right )$.  Hence, we can also arrange that $p(E) \ge 1 -\epsilon$ (which follows of course from generic property in Proposition~\ref{setgensmall} too).
\end{remark}

The constructions above raised the question of how slowly we can arrange $m(\bigcup\limits_{k=1}^n \tau^{-k} (E))$ to grow.  In fact, we can get this to grow as slowly as we like.   To show this, we used this consequence of the Rokhlin Lemma.  This lemma was also an important feature in some of the arguments in del Junco and Rosenblatt~\cite{dJR}; see the corresponding lemma in that paper.

\begin{lemma}\label{AI}  Suppose $\epsilon > 0$, $0 < \delta < 1$, and $n \ge 1$.  Then there is a set $A\in \mathcal B$ such that $m(A) =
\delta$,
and $m\left (\bigcap\limits_{k=1}^n \tau^k(A) \cap A\right ) \ge (1 - \epsilon)m(A)$.
\end{lemma}

Now here is the proposition we need to get slowly sweeping out sets.  This is a variation on the one that appears in ~\cite{AR}.

\begin{prop} \label{gottheslowrate} Consider a sequence $(\epsilon_n: n \ge 1)$ with $0 < \epsilon_n < 1/2$ for all $n$, and $\epsilon_n$ decreasing to $0$.  Then there exists $E$ such that $m(E) > 0$ and  $m(\bigcup\limits_{k=1}^n \tau^k(E)) \le 1-\epsilon_n$ for all $n$.
\end{prop}
\begin{proof} We will construct an increasing  sequence $(N_j)$ and a decreasing sequence  $(\delta_j)$ with certain properties. Let $\gamma_1 = 2\epsilon_1$ and $N_1 = 1$.  Choose $\gamma_j, j \ge 2$ decreasing to zero with $\sum\limits_{j=2}^\infty \gamma_j < 1 - 2\epsilon_1$.  So $\sum\limits_{j=1}^\infty \gamma_j < 1$.  Let $N_j, j\ge 2$ be an increasing sequence such that $2\epsilon_{N_j} \le \gamma_j$ for all $j \ge 2$.  As in Lemma~\ref{AI}, we can construct $(A_j: j \ge 1)$ such that $m(A_j) = \gamma_j$ and $m(\bigcap\limits_{k=1}^{N_{j+1}} \tau^k(A_j)\cap A_j) \ge \frac 12 m(A_j)$ for all $j \ge 1$.  Now let $E = X\backslash \bigcup\limits_{j=1}^\infty A_j$.  We have $m(E) > 0$.
Let $M \ge 1$.  There is some $j\ge 1$ such that $N_j \le M < N_{j+1}$.  So
\begin{align*}
&m(\bigcup\limits_{k=1}^M \tau^k(E))\le m(\bigcup\limits_{k=1}^M \tau^k(X\backslash A_j))\le 1 - m(\bigcap\limits_{k=1}^M \tau^k(A_j)) \\
&\le 1 - m(\bigcap\limits_{k=1}^{N_{j+1}} \tau^k(A_j))\le 1 - \frac 12 m(A_j) = 1 - \frac 12\gamma_j \le 1 - \epsilon_{N_j} \le 1 - \epsilon_M.
\end{align*}
\end{proof}

Proposition~\ref{gottheslowrate} shows that if we take $B_M = X\backslash \bigcup\limits_{k=1}^M \tau^k(E)$, then $(B_M: M\ge 1)$ is a shrinking target with $\tau^k(E)$ and $B_M$ disjoint for all $k=1,\dots,M$.  If we started this construction only for somewhat larger $n$, then this would allow us to have $m(E)$ as close to $1$ as we like and also have the values $m(B_M)$ of the shrinking target decreasing to zero as slowly as we like, but still $E$ is not visible to $(B_M)$ i.e. $\tau^M(x) \notin B_M$ for all $x\in E$.  Indeed, $\tau^k(E) \cap B_M = \emptyset$ for all $k = 1,\dots,M$.

In fact, a slight modification of the construction in Proposition~\ref{gottheslowrate} gives one of the most important results of
this section of the paper.

\begin{cor}\label{slowshrink} Fix $\tau$ ergodic, and any $(\epsilon_n: n \ge 1)$ decreasing to $0$, with $0 <\epsilon_n$ for all $n$.  There is an a.e. invisible measure-theoretic shrinking target $(B_n)$ with respect to $\tau$ such that $m(B_n) \ge \epsilon_n$ for all $n$.
\end{cor}
\begin{proof}  Proceed as in the proof of Proposition~\ref{gottheslowrate}.  Take $M \ge 1$ and $j$ such that $N_j\le M < N_{j+1}$.  Let $E_M = X\backslash\bigcup\limits_{l=j}^\infty A_l$.  Then it is still the case that $m(\bigcup\limits_{k=1}^M\tau^k(E_M)) \le 1 -\epsilon_M$.   But now take $B_M = X\backslash \bigcup\limits_{k=1}^M \tau^k(E_M)$.  Then $(B_M)$ is a shrinking target with
$m(B_M) \ge \epsilon_M$ for all $M$.  Moreover, $\tau^k(E_M)$ and $B_M$ are disjoint for all $k = 1,\dots,M$.  Since $m(E_M)$ increases to $1$ as $M\to \infty$, $(B_M)$ is an invisible shrinking target.
\end{proof}

\begin{remark} In Corollary~\ref{slowshrink}, we can of course arrange in particular $\sum\limits_{n=1}^\infty m(B_n) = \infty$ since the measures $m(B_n)$ can be made to shrink as slowly as we like.
\end{remark}

But now also suppose one is given a sequence $(B_n)$ as in Proposition~\ref{slowvisible}.  There is then $\sigma \in \mathcal M$, such that $\sigma (B_n) = [0,\epsilon_n]$ for all $n$.  Hence, the intervals $([0,\epsilon_n]: n\ge 1)$ are invisible with respect to $\omega = \sigma\circ\tau\circ\sigma^{-1}$, a mapping that is isomorphic to $\tau$.  This gives

\begin{cor}\label{isominvisible} Suppose $(\epsilon_n: n \ge 1)$ is decreasing and $\tau$ is ergodic.  Then there is an ergodic map $\omega$ that is isomorphic to $\tau$ such that  $([0,\epsilon_n]: n \ge 1)$ is an a.e. invisible measure-theoretic shrinking target with respect to  $\omega$.
\end{cor}

On the other hand, if we avoid the obvious obstruction on the $m(B_n)$, we can also take a fixed $\tau$ and construct a.e. visible measure-theoretic shrinking targets with any desired rate of decreasing measure.  For the proof of this, it is useful to first prove this lemma.

\begin{lemma}\label{helpAI} Consider an ergodic map $\tau$ and $\gamma_l, 1 \le l \le L$,
with $0 < \gamma_L <\dots < \gamma_1 <1$.  For any $\eta > 0$, there are decreasing sets $U_l$ with $m(U_l) = \gamma_l$ for $1\le l \le L$, such that the sets $\tau^{l-1}(U_l), 1\le l \le L$ are mutually independent with overall error at most $\eta$.
\end{lemma}
\begin{proof} As usual we denote $E^1 = E$ and $E^c = X\backslash E$.  We want to construct nested sets $(U_l)$ so that for
any choice of $e_l \in \{1,c\}, 1 \le l \le L$, we have
\[\left |m(\bigcap\limits_{l=1}^L \tau^{l-1}(U_l^{e_l})) - \prod\limits_{l=1}^L m(U_l^{e_l})\right| \le \eta.\]
To achieve this, we take the template given by some Bernoulli map $\sigma$ realized as the coordinate shift on $\prod\limits_{-\infty}^\infty [0,1]$ in the product probability measure with coordinate probability measure $p$ being Lebesgue measure on $[0,1]$.  We take sets $W_l$ defined by restricting the base coordinate to $[0,\gamma_l]$ and leaving all other coordinates unrestricted in $[0,1]$.
The powers $\sigma^{l-1}(W_l)$ are independent sets.  Now select some measure-preserving map $\omega$ so that $\omega\circ \tau\circ \omega^{-1}$ is extremely close to $\sigma$ in the weak topology.   Then $\omega(\tau^{l-1}(\omega^{-1}(W_l)))$ is very close in the weak topology to $\omega(\sigma^{l-1}W_l)$ simultaneously for all $l$.  Let $U_l = \omega^{-1}(W_l)$.  These sets are nested with the desired measures.  Also, the independence of $\sigma^{l-1}(W_l),1\le l \le L$ gives us the near independence of
$\omega(\tau^{l-1}(U_l)),1\le l \le L$, and hence of course the same degree of near independence of $\tau^{l-1}(U_l),1\le l \le L$.
\end{proof}

\begin{remark}\label{adjustment}  Lemma~\ref{helpAI} can be adapted to giving a similar conclusion scaled into a tall Rokhlin tower on $\tau$.  The simplest method for this is to take the tower and map the top of the tower to the bottom so that the new mapping is ergodic.
\end{remark}

Here now is a second important result of this section of the paper.

 \begin{prop}\label{slowvisible} Suppose $(\epsilon_n:n\ge 1)$ is decreasing and $\sum\limits_{n=1}^\infty \epsilon_n = \infty$.  Suppose $\tau$ is ergodic.  Then there is an a.e. visible shrinking target $(B_n)$ such that $m(B_n) = \epsilon_n$ for all $n$.
\end{prop}
\begin{proof}  Consider pairwise disjoint blocks of terms  $I_j$ in $\mathbb N$, chosen so that $\Sigma_j = \sum\limits_{n\in I_j} m(B_n) \ge j$ for all $j\ge 1$.   Then by the standard argument used in the strong direction of the Borel-Cantelli Lemma, if the $\tau^{-n}(B_n)$ are independent on the blocks $I_j$, then $m(\bigcap\limits_{n\in I_j} X\backslash \tau^{-n}(B_n)) \le e^{-j}$ for all $j$.  Hence, since $\sum\limits_{j=1}^\infty e^{-j} < \infty$,  a.e. $x$ is only in finitely many of these intersections.  That is, for a.e. $x$, if $j$ is large enough, we have $\tau^{n}(x) \in B_n$ for some $n \in I_j$.  It follows that $(B_n)$ is a.e. visible with respect to $\tau$.

We actually construct the sets $(B_n)$ only so that $\tau^{-n}(B_n)$ are approximately independent on the blocks $I_j$, but with a good enough approximation so that the same conclusion as above can be made.  Also, the blocks $(B_n,n \in I_j)$ will be constructed inductively.  We arrange that the sets $(B_n:n\in I_j)$ are decreasing and $m(B_n) = \epsilon_n$ for $n\in I_j$.  But the sets $(B_n)$ for indices not in a block $I_j$ are not explicitly determined; these are just chosen so that overall the sequence $(B_n)$ is shrinking and $m(B_n) = \epsilon_n$ for all $n\ge 1$.  For this insertion of additional  sets $B_n$ to work, we will also have to arrange that if $I_j = [a_j,b_j]$, then $B_{a_{j+1}} \subset B_{b_j}$ for all $j$.

First, take a Kakutani skyscraper for $\tau^{-1}$ built on a base $B$ with $m(B) \in (0,1)$.  Assume for simplicity that $m(\bigcup\limits_{k=1}^n \tau^{-k}(B)) < 1$ for all $n$.  It is not necessary that $B$ have small measure.  It is important only that by ergodicity the skyscraper has arbitrarily high levels of first time return to $B$ so that our inductive choice of the blocks $(B_n: n \in I_J)$ can proceed.   At each stage of the construction, we will choose the sets $(B_n: n \in I_J)$ with the desired approximate independence and such that the appropriate nesting occurs.

To be specific, let $D_k \subset B$ be the set of $x \in D_k$ such that the smallest value of $n$ such that $\tau^{-n}(x) \in B$ is $n= k$.  Let $R_k = \bigcup\limits_{n=1}^{k-1} \tau^{-n}(D_k)$.  Then $X$ is partitioned by the sets $R_k$ and the union of them is the Kakutani skyscraper.

We will be selecting the sets $B_n$ for $n\in I_j$ with the following properties.  There will be increasing sequences $(s_j)$ and $(t_j)$ such that $s_j < t_j < s_{j+1}$ for all $j$.  The sets $B_n$ for $n \in I_j$ will be chosen so that they are subsets of $\bigcup\limits_{k=s_j}^\infty R_k$ while also each such $B_n$ will contain the entire tail $\bigcup\limits_{k=t_j+1}^\infty R_k$.  This is what guarantees the nesting property needed between the blocks i.e. $B_{a_{j+1}} \subset B_{b_j}$ for all $j$.

We need to see how to get the approximate independence of the sets $\tau^{-n} (B_n),n \in I_j$ under these constraints.  First, we construct $B_n^1,n\in I_j$ each to be a subset of $\bigcup\limits_{k=s_j}^{t_j} R_k$.  We also choose them
nested with $m(B_n^1) = \epsilon_n -\delta$ for all $n\in I_j$.  Here $\delta$ will be chosen appropriately small enough. We also let $B_n^2,n\in I_j,$ each be the entire tail $\bigcup\limits_{k=t_j+1}^\infty R_k$.  By inductively adjusting the choices of $s_j$ and $t_j$, we can guarantee that the choice of $m(B_n^2) = \delta$ for all $n\in I_j$ is small enough so that any approximate independence we have achieved with $\tau^{-n}(B_n^1), n \in I_j,$ will be perturbed very little when we replace $B_n^1$ by $B_n^1 \cup B_n^2$ for all $n\in I_j$.
Also, the choices of $(s_j,t_j)$ and $(a_j,b_j)$ will be large enough, so that we have the block sums $\sum\limits_{n\in I_j} \epsilon_n \ge j$.

Assume that the first $J$ blocks $(B_n:n \in I_j), j=1,\dots,J,$ and the first $J$ pairs $(s_j,t_j), j =1,\dots,J,$  have been constructed with the guidelines above.  We now choose $s_{j+1} > t_j$ and then $t_{j+1} > s_{j+1}$ so that $m(\bigcup_{k=t_{j+1}+1}^\infty R_k) = \delta$ is extremely small relative to the $m(\bigcup\limits_{k=s_{j+1}}^{t_{j+1}} R_k)$.  We also arrange that $I_{j+1}$ has $a_{j+1}$ large enough so that $\epsilon_n, n \in I_{j+1}$ are all much less than $m(\bigcup\limits_{k=s_{j+1}}^{t_{j+1}} R_k)$.  That is,  $\epsilon_{b_{j+1}}$ is much less than $m(\bigcup\limits_{k=s_{j+1}}^{t_{j+1}} R_k)$.  We have no obstacle to increasing the heights of the towers $R_k$ by increasing the choice of $s_{j+1}$.  Therefore, using Lemma~\ref{helpAI} (see also Remark~\ref{adjustment}) scaled appropriately and applied to the map $\tau^{-1}$ on the towers $R_k$, this allows us to create measure-theoretically decreasing sets  $B_n^1,n\in I_j$, with $m(B_n^1) = \epsilon_n -\delta$, as subsets of $\bigcup\limits_{k=s_{j+1}}^{t_{j+1}} R_k$ in such a fashion that $\tau^{-n}(B_n), n \in I_{j+1},$ are extremely close to being mutually independent.    An inspection of the process above shows, that with appropriate choices of $I_{j+1}$ and $(s_{j+1},t_{j+1})$, all of the desired properties of $B_n, n\in I_{j+1}$ can be achieved simultaneously.
\end{proof}

Suppose one is given a sequence $(B_n)$ as in Proposition~\ref{slowvisible}.  There is then $\sigma \in \mathcal M$, such that $\sigma (B_n) = [0,\epsilon_n]$ for all $n$.  Hence, the intervals $([0,\epsilon_n]: n\ge 1)$ are a.e. visible with respect to $\omega = \sigma\circ \tau\circ \sigma^{-1}$, a mapping that is isomorphic to $\tau$.

\begin{cor}\label{isomaevisible}  Suppose $(\epsilon_n:n\ge 1)$ is decreasing and $\sum\limits_{n=1}^\infty \epsilon_n = \infty$.  Suppose $\tau$ is ergodic.  Then there is an ergodic map $\omega$ that is isomorphic to $\tau$, such that  $([0,\epsilon_n]: n\ge 1)$ is an a.e. visible measure-theoretic shrinking target with respect to $\omega$.
\end{cor}

\begin{remark} Corollary~\ref{isominvisible} and Corollary~\ref{isomaevisible} give interesting contrasts and show how at least visibility of  measure-theoretic shrinking targets is about the nature of the map with respect to the shrinking target, and not something intrinsic in the shrinking target by itself.
\end{remark}

\noindent {\bf Acknowledgments}: We thank J. Chaika for very useful conversations about shrinking targets and IETs.  We thank D. Kleinbock for pointing out a number of important references that we had not originally included in a literature review for this article.  We also thank M. Wierdl and A. Parrish for suggestions on how to make the construction in Proposition~\ref{slowvisible} work by using the Kakutani skyscraper construction.

While this article was in the final stages of preparation for publication, we learned the sad news that Misha Boshernitzan passed away.  We would like to dedicate this article to him, in memory of all the wonderful mathematics that he created and the encouragement that he gave to others to do the same.

\end{document}